\documentclass{amsart}
\usepackage{amsmath,amssymb}
\newtheorem{theorem}{Theorem}    
 
\newtheorem{claim}{Claim}

\newtheorem{question}{Question} 

\theoremstyle{definition}
\newtheorem{definition}[theorem]{Definition}

\newtheorem*{remark*}{Remark}
\newcommand{\Z}{\mathbb{Z}}

\newcommand{\ent}{\mathrm{ent}}
\newcommand{\ex}{\mathrm{ex}}

\begin{document}

\title[Infinitely many non-conjugate braids]{A non-degenerate exchange move always produces infinitely many non-conjugate braids }
\author[T.Ito]{Tetsuya Ito}
\address{Department of Mathematics, Kyoto University, Kyoto 606-8502, JAPAN}
\email{tetitoh@math.kyoto-u.ac.jp}

\begin{abstract}
We show that if a link $L$ has a closed $n$-braid representative admitting non-degenerate exchange move, an exchange move that does not obviously preserve the conjugacy class, $L$ has infinitely many non-conjugate closed $n$-braid representatives.
\end{abstract}

\maketitle


Let $B_n$ be the braid group with standard generators $\sigma_1,\ldots,\sigma_{n-1}$. We denote the closure of a braid $\beta \in B_n$ by $\widehat{\beta} \subset S^{3}$. For an oriented link $L \subset S^{3}$ let $\mathbf{Br}_n(L)= \{\beta \in B_n \: | \:\widehat{\beta}=L\}$ be the set of $n$-braids whose closures are $L$.

The set $\mathbf{Br}_n(L)$ may contain infinitely many mutually non-conjugate braids. However, Birman and Menasco proved a remarkable (non)finiteness theorem \cite{BM}: $\mathbf{Br}_n(L)$ modulo \emph{exchange move} $\alpha \sigma_{n-1}\beta \sigma_{n-1}^{-1} \leftrightarrow \alpha \sigma_{n-1}^{-1}\beta \sigma_{n-1}$ ($\alpha,\beta \in B_{n-1}$) has only finitely many conjugacy classes. 
In particular, when $\mathbf{Br}_n(L)$ does not contain a braid admitting exchange move, $\mathbf{Br}_n(L)$ contains only finitely many conjugacy classes. 

Then we ask the converse\footnote{We need to be bit  careful to formulate the problem since some exchange moves are `trivial' in the sense that they obviously yield conjugate braids.}:
 \emph{Does $\mathbf{Br}_n(L)$ contain infinitely many mutually non-conjugate braids if $\mathbf{Br}_n(L)$ contains a braid admitting exchange move ?}

This question was studied in \cite{SS,Sh,St1,St2} where it was shown that under some additional and technical assumptions, iterations of exchange move indeed produce infinitely many non-conjugate braids. 

In this note we give a simpler and shorter proof of infiniteness under the weakest assumption. We use a formulation of iterations of exchange moves following \cite{SS}:
\begin{definition} 
We say that an $n$-braid $\beta$ \emph{admits an exchange move} if one can write $\beta=AB$ for $A \in \langle \sigma_1^{\pm 1},\ldots,\sigma_{n-2}^{\pm 1}\rangle$ and $B \in \langle \sigma_2^{\pm 1},\ldots,\sigma_{n-1}^{\pm 1}\rangle$. 

For $k \in \Z$ and an $n$-braid $\beta=AB$ admitting an exchange move, the \emph{$k$-iterated exchange move} of $\beta=AB$ is the braid $\ex^{k}(\beta)= A\tau^{k}B\tau^{-k}$ where $\tau=(\sigma_2\cdots \sigma_{n-2})^{n-2} \in B_n$.
We say that an (iterated) exchange move is \emph{degenerate} if $A \tau = \tau A$ or $B \tau = \tau B$. Otherwise, an (iterated) exchange move is \emph{non-degenerate}.

\end{definition}

A $k$-iterated exchange move is attained by exchange move $|k|$ times so the closures of $\beta$ and $\ex^k(\beta)$ represent the same link. A degenerate exchange move is an exchange move that obviously preserves the conjugacy classes. Our main theorem shows that, except this trivial cases, iterated exchange moves \emph{always} produce infinitely many mutually non-conjugate braids.

We identify the braid group $B_n$ with the mapping class group $MCG(D_n)$ of the $n$-punctured disk $D_n$. Let $\ent(\beta)$ be the topological entropy of $\beta$, the infimum of the topological entropy of homeomorphism that represents $\beta$. 

\begin{theorem}
\label{theorem:main}
If $\beta \in \mathbf{Br}_n(L)$ admits a non-degenerate exchange move, then the set $\{\ent(\ex^k(\beta)) \: | \: k \in \Z\}$ is unbounded. In particular, the set  $\{\ex^{k}(\beta) \: | \: k \in \Z\} \subset \mathbf{Br}_n(L) $ contains infinitely many distinct conjugacy classes.
\end{theorem}

Let $S$ be a closed orientable surface minus finitely many disks and puncture points in its interior.
A simple closed curve $c$ in $S$ is \emph{essential} if $c$ is not boundary parallel nor surrounding single puncture. We denote by $T_c$ the Dehn twist along $c$. A family of essential simple closed curves $\{c_1,\ldots,c_N\}$ \emph{fills} $S$ if $\min i(c,c_i)\neq 0$ for any essential simple closed curve $c$, where $i(c,c')$ denotes the geometric intersection number. Our proof is based on the following theorem of Fathi \cite[Theorem 7.9]{Fa}.

\begin{theorem}
\label{theorem:Fathi}
Let $f \in MCG(S)$ and $c_1,\ldots,c_N$ be essential simple closed curves in $S$.
Assume that
\begin{enumerate}
\item[(i)] The set of curves $\{c_1,\ldots, c_N\}$ fills $S$.
\item[(ii)] $i(c_i,c_{i+1})\neq 0$ for $i=1,\ldots,N-1$ and $i(c_{N},c_1)\neq 0$.
\end{enumerate} 
Then for given $R>0$, there is $k=k(R)>0$ such that $ T_{c_1}^{n_{1}}T_{c_{2}}^{n_{2}}\cdots T_{c_N}^{n_N}f  $
is pseudo-Anosov whose dilation is $>R$ whenever $|n_i|>k$ for all $i$.

\end{theorem}

\begin{remark*}
Although the statements and assumptions of Theorem 
\ref{theorem:Fathi} are bit different,  Theorem \ref{theorem:Fathi} follows from the proof of \cite[Theorem 7.9]{Fa}; First, the assumptions (i) and (ii) allow us to apply an interpolation inequality \cite[Theorem 7.4]{Fa}. Second, to get dilatation bound we take a choice of $\varepsilon>0$ in page 149 of the proof of \cite[Theorem 7.9]{Fa} as $\varepsilon=(K^{2}R^{2l-1})^{-1}$ instead of $(2K^{2})^{-1}$ in the original argument. Then the same argument gives the desired dilatation bound.
\end{remark*}

\begin{proof}[Proof of Theorem \ref{theorem:main}]
The braid $\tau$ in iterated exchange move corresponds to the Dehn twist $T_{c}$ along the simple closed curve $c$ surrounding $2$nd,...,($n-2$)th punctures. The non-degeneracy assumption is equivalent to saying that $A(c)\neq c$ and $B(c) \neq c$.

For $i>0$ let $c_{2i}=\beta^{i-1}(A(c))=(AB)^{i-1}A(c), c_{2i-1}=\beta^{i-1}(c)=(AB)^{i}(c)$. Then for $N>0$, 
$\ex^{k}(\beta)^{N} = T_{c_1}^{k}T_{c_2}^{-k}\cdots T_{c_{2N-1}}^{k}T_{c_{2N}}^{-k}\beta^{N}$.

We equip a complete hyperbolic metric (of finite area) on $D_n$ and take $c_i$ as a closed geodesic.
Let $S$ be the minimum complete geodesic subsurface of $D_n$ that contains $\{c_1,c_2,c_3,c_4,\ldots\}=\{c,A(c), \beta(c),\beta(A(c)),\ldots\}$. Here the minimum means the minimum with respect to inclusions. Then for sufficiently large $M$ the set of curves $\{c_1,c_2,\ldots,c_{M}\}$ fills $S$. By non-degeneracy assumption, $S$ contains $\partial D_n$ as its boundary.

\begin{claim}
\label{claim:A}
$\beta$ preserves the subsurface $S$ setwise. In particular, the restriction $\ex^{k}(\beta)|_S \in MCG(S)$ is well-defined for all $k$.
\end{claim}
\begin{proof}[Proof of Claim \ref{claim:A}]
If $S=D_n$ then it is obvious so we assume that $S \neq D_n$. Then $C=\partial S \setminus \partial D_n$ is a non-empty multi curve. If $\beta(S) \neq S$ then $\beta^{-1}(S) \neq S$ and $i(\beta^{-1}(C),C) \neq 0 $. Since $\{c_1,c_2,\ldots,c_{M}\}$ fills $S$ this means $i(\beta^{-1}(C),c_i)=i(C,\beta(c_i)) = i(C,c_{i+2})\neq 0$ for some $i$, this is a contradiction since $i(C,c_{i+2})=0$ by definition.
\end{proof}

\begin{claim}
\label{claim:B}
There exists $N\geq M$ such that $i(c_{1},c_{2N})\neq 0$.
\end{claim}
\begin{proof}[Proof of Claim \ref{claim:B}]
Assume to the contrary that $i(c_{1},c_{2N}) =0$ for all $N\geq M$. Let $S'$ be the minimum geodesic subsurface of $S$ that contains $\{c_{2M},c_{2(M+1)},\ldots\}=\{c_{2M},\beta(c_{2M}),\beta^{2}(c_{2M}),\ldots\}$. By the same argument as Claim \ref{claim:A}, $\beta$ preserves $S'$, and $i(c,c')=0$ for every curve $c' \subset S'$. Then $\beta^{-2(M-1)}(c_{2M}) = A(c) \subset S'$ so $i(c,A(c))=0$. This contradicts with non-degeneracy assumption.
\end{proof}

By non-degeneracy assumption $i(c_{i},c_{i+1})\neq 0$ for every $i>0$. Thus by Claim \ref{claim:B} there is $N>0$ such that such 
\[ \ex^{k}(\beta)^{N}|_S = T_{c_1}^{k}T_{c_2}^{-k}\cdots T_{c_{2N-1}}^{k}T_{c_{2N}}^{-k}\beta^{N}|_S \in MCG(S)\] satisfies the assumptions of Theorem \ref{theorem:Fathi}.
Hence for any given $R>0$, whenever $|k|$ is sufficiently large, $\ex^{k}(\beta)^{N}|_S$ is pseudo-Anosov whose dilatation $\lambda(\ex^{k}(\beta)^{N}|_S)$ is $>R$. 
Since \begin{align*}
\ent( \ex^{k}(\beta)) & = \frac{\ent ( \ex^{k}(\beta)^{N} )}{N} \geq  \frac{\ent( \ex^{k}(\beta)^{N}|_S)}{N} =  \frac{\log \lambda(\ex^{k}(\beta)^{N}|_S)}{N} \geq \frac{\log R}{N} \end{align*}
 the set $\{\ent(\ex^{k}(\beta)) \: | \: k \in \Z\}$ is unbounded. 
\end{proof}

Our proof shows that as $k$ increase, a non-degenerate $k$-iterated exchange move increases the entropy, the complexity of dynamics, as long as $k$ is sufficiently large. Since the core of Birman-Menasco's proof of (non)finiteness theorem is to reduce the complexity (the number of singular points) of braid foliation of Seifert surface, it is natural to expect relations between the entropy and braid foliation. 

\begin{question}
If a braid $\beta$ is obtained from $\beta'$ by an exchange move reducing the complexity of braid foliation, then $\ent(\beta) \leq \ent(\beta')$ ?
\end{question} 

\section*{Acknowledgement}
The author is partially supported by JSPS KAKENHI Grant Numbers19K03490, 16H02145.  This research was supported in part by funding from the Simons
Foundation and the Centre de Recherches Math\'ematiques, through the Simons-CRM
scholar-in-residence program.

\end{document}